\documentclass[12pt,a4paper]{article}
\usepackage[utf8]{inputenc}

\usepackage{amsmath}
\usepackage{amsfonts}
\usepackage{amssymb}
\usepackage{amsthm}
\usepackage{mathrsfs}
\usepackage{makeidx}
\usepackage[left=2.5cm,right=2.5cm,top=2.5cm,bottom=2.5cm]{geometry}
\usepackage{setspace}
\usepackage[all]{xy}

\DeclareMathOperator{\rad}{\text{rad}}

\DeclareMathOperator{\Ker}{\text{Ker}}
\DeclareMathOperator{\Aa}{\mathcal{A}}
\DeclareMathOperator{\dimk}{\operatorname{dim}_k}

\theoremstyle{definition}
\newtheorem{defi}{Definition}[section]
\newtheorem{lem}[defi]{Lemma}
\newtheorem{teo}[defi]{Theorem}

\newtheorem{cor}[defi]{Corollary}
\newtheorem{ex}[defi]{Example}

\theoremstyle{remark}
\newtheorem{obs}[defi]{Remark}

\title{\textbf{On the correspondence between path algebras and generalized path algebras}}
\author{Viktor Chust\textsuperscript{a,*} \and Flávio Ulhoa Coelho\textsuperscript{b}\\ \small{\textsuperscript{a,b}Institute of Mathematics and Statistics - University of São Paulo, São Paulo, Brazil}\\
\small{\textsuperscript{*}Corresponding author. E-mail address: viktorch@ime.usp.br}\\
\small{ORCID: 0000-0003-4931-4222}}
\date{}

\begin{document}

\maketitle

\begin{abstract}
The concept of generalized path algebras was introduced in (Coelho and Liu, 2000). 
It was shown in (Ibáñez Cobos et al., 2008) how to obtain the Gabriel quiver of a given generalized path algebra. In this article, we generalize the concept of generalized path algebra to allow them to have relations, and we extend the result in (Ibáñez Cobos et al., 2008) to this new setting. 
Moreover, we use the extended result mentioned above to address the inverse problem: that is, the problem of determining when a given algebra is isomorphic to a generalized path algebra in a non-trivial way. \\[1ex]
Keywords: generalized path algebras, bound path algebras, representations of algebras \\[1ex]
Mathematics Subject Classification: Primary 16G10, Secondary 16G20
\end{abstract}

\section{Introduction}
\label{sec:intro}

Along this paper, all algebras are assumed to be finite dimensional basic $k$-algebras, where $k$ is an algebraically closed field. For such an
algebra, there is a standard and already classical way of \textit{representing} it by means of a quiver, which is given by a theorem due to Gabriel and the so-called \textbf{path algebras}. Namely, given an algebra $A$, there exists a (uniquely determined) quiver $Q_A$ such that $A$ is isomorphic to a quotient of the path algebra $k Q_A$ (see Subsection~\ref{subsec:gabriel} for details). 

Also, as introduced by F.U. Coelho and S.X. Liu in \cite{CLiu}, the 
concept of \textbf{generalized path algebras} yields another way of constructing algebras from a quiver  (see Section~\ref{sec:noc_prem} below for definitions and basic results). The underlying idea behind the concept of generalized path algebras is this: given a quiver $Q$, assign to each of its vertices $x$ an algebra  $A_x$ (in the classical case of path algebras, it is assigned the base field $k$). Then the multiplication is induced by composition of paths and multiplication inside the $A_x$. 

The generalized path algebras are particular cases of tensor algebras over certain pro-species, an idea introduced by J. Külshammer in \cite{Kuls17}. We also refer to \cite{ICNLP,FLi1} for discussions on the relationship between generalized path algebras and other concepts.

Much like an usual path algebra may have relations, we can also consider quotients of path algebras and generalized path algebras by suitable ideals, giving the \textbf{bound path algebras} and \textbf{generalized bound path algebras} (here also called \textbf{gbp-algebras}) respectively.

An algebra $A= kQ_A/ I_A$ (where $I_A$ is an admissible ideal) can naturally be seen as a gbp-algebra in two different ways. On the one hand, using its ordinary quiver (and relations) and the usual Gabriel construction for path algebras mentioned above, and, on the other, using a quiver with a sole vertex and no arrows and the algebra itself assigned to it. We shall call these ways of representing $A$ as a gbp-algebra as \textbf{trivial}. A question which naturally arises is on the possibility of having descriptions other than the above ones for $A$ as a gbp-algebra. Clearly, if this is possible, then to $A$ we shall assign a quiver (generally smaller than its Gabriel quiver) and a set of algebras, one for each vertex of this new quiver, and this might, in principle, allow us to better understand the original algebra. This connection will be our main focus of discussion in this paper.

In Section 2, we recall some basic notions on quivers and algebras and establish the needed notations. Section 3 is devoted to prove a result (namely, Theorem~\ref{th:gen_icnlp}) which establishes the ordinary quiver and the relations of a given gbp-algebra, thus extending a result from \cite{ICNLP}. 

The remaining of the paper will then be devoted to the inverse process: how to obtain a (non-trivial) gbp-algebra isomorphic to a given bound path algebra $A$. This is not always possible and, in case it is, we shall call it a \textbf{non-trivial simplification of $A$}. As we shall see, the non-trivial simplifications of $A$ can be obtained by determining which are the equivalence relations on the vertices of the Gabriel quiver of $A$ that satisfy some combinatorial properties. These discussions are summarized in Theorems~\ref{th:simpl_from_equiv_rel} and~\ref{th:equiv_rel_from_simpl}. 

In a forthcoming paper we shall discuss the representations and some homological properties of the generalized (bound) path algebras \cite{Chust2}.

\section{Definitions and Basic Properties}
\label{sec:noc_prem}

In this preliminary section, we fix some notation and state some of the basic ideas that shall be used through these pages. Then we shall properly recall the definition of generalized path algebras and give some basic properties of these.

If the reader wants a basic reference covering the topics treated in this section, we recommend \cite{AC,ASS,ARS}.

\subsection{\textit{Quivers}}

It will be useful to fix some notations and assumptions regarding quivers:

\begin{defi}
A \textbf{quiver}\index{quiver} is a tuple $Q = (Q_0,Q_1,s,e)$, where $Q_0$ and $Q_1$ are sets and $s,e: Q_1 \rightarrow Q_0$ are functions. The elements of $Q_0$ are called \textbf{vertices}\index{vertex}, the elements of $Q_1$ are called \textbf{arrows}\index{arrow}, and given an arrow $\alpha \in Q_1$, $s(\alpha)$ is called \textbf{start}\index{arrow!start} of $\alpha$, and $e(\alpha)$ is called \textbf{end}\index{arrow!end} of $\alpha$.
\end{defi}

It will be useful to adopt the following notations: given a quiver $Q$ and two vertices $i,j \in Q_0$, we denote $Q(i,j) = \{\alpha \in Q_1: s(\alpha) = i \text{ and } e(\alpha)=j\}$ and indicate by $[i,j]_Q$ the number of arrows belonging to $Q(i,j)$. (Which means that $[i,j]_Q = |Q(i,j)|$).

If $\alpha$ is an arrow of $Q$ and $s(\alpha) = e(\alpha)$, then $\alpha$ is said to be a \textbf{loop}\index{loop}.

In this paper, we shall be always assuming that the quivers are \textbf{finite}\index{quiver!finite}, which means that $Q_0$ and $Q_1$ will be finite sets.

\begin{defi}
Let $Q$ be a quiver. To each vertex $i \in Q_0$ of $Q$ we associate a \textbf{zero-length path} over $Q$, which is denoted by $\epsilon_i$. We also say that $\epsilon_i$ starts and ends at $i$ and denote $\epsilon_i: s(\epsilon_i)=i \rightsquigarrow e(\epsilon_i)=i$. 

A \textbf{path} of length $t$ over $Q$ is a finite sequence $\gamma = (\gamma_1,\ldots,\gamma_t)$, where $t \in \mathbb{N}$, and $\gamma_1,\ldots,\gamma_t \in Q_1$ are arrows of $Q$, such that for every $i > 1$, $s(\gamma_{i+1}) = e(\gamma_i)$. The vertices $s(\gamma_1)$ and $e(\gamma_t)$, respectively, are called \textbf{start}\index{path!start} and \textbf{end}\index{path!end} of the path $\gamma$. We shall also use the following notation for paths: $\gamma = \gamma_1 \ldots \gamma_t : s(\gamma) \rightsquigarrow e(\gamma)$. When necessary to distinguish this concept from that of $\Aa$-paths to be defined below, we shall be calling $\gamma$ an \textbf{ordinary path} over $Q$. 

In all cases, we can denote the length of a path $\gamma$ by $l(\gamma)$.

Also, a path $\gamma$ is called an \textbf{oriented cycle}\index{oriented cycle} if it starts and ends at the same vertex. In the case where $Q$ has no oriented cycles, $Q$ is said to be an \textbf{acyclic}\index{quiver!acyclic} quiver.
\end{defi}

\subsection{\textit{Category of quivers}}
\label{subsec:cat_quiv}

There is a category, denoted by $\textbf{Quiv}$, whose class of objects is the class of all quivers. Given two quivers $Q = (Q_0,Q_1,s,e)$ and $Q' = (Q'_0,Q'_1,s',e')$, a \textbf{quiver morphism}\index{quiver!morphism} $f: Q \rightarrow Q'$ is a pair $f = (f_0,f_1)$, where $f_0: Q_0 \rightarrow Q'_0$ and $f_1: Q_1 \rightarrow Q'_1$ are functions satisfying $s' \circ f_1 = f_0 \circ s$ and $e' \circ f_1 = f_0 \circ e$. Note that, if $f_0$ and $f_1$ are surjective, then $f$ is an epimorphism in the categorical sense.

It will be worth noting that the category $\textbf{Quiv}$ has quotient objects. Let $Q$ be a quiver and let $\sim \subseteq Q_0 \times Q_0$ be an equivalence relation on the vertices of $Q$. Then we define the \textbf{quotient quiver}\index{quiver!quotient} $\overline{Q} = \frac{Q}{\sim}$. The set of vertices of $\overline{Q}$ is the quotient set $\frac{Q_0}{\sim}$, and given $\overline{a},\overline{b} \in \overline{Q}_0$, the number of arrows between these vertices is given by the following formula:

$$[\overline{a},\overline{b}]_{\overline{Q}} = \max \{[x,y]_Q: x \in \overline{a}, y \in \overline{b}\}$$ 

\subsection{\textit{Path Algebras}}
\label{subsec:path_algebras}

Here we will recall the usual concept of path algebra, since its generalization will be discussed below.

Let $Q$ be a quiver. Let $kQ$ be the $k$-vector space having as basis the set of paths over $Q$. We want to define an internal multiplication in $kQ$. By linearity, it is sufficient to define what is the product between two paths. This, by its turn, is given by the \textbf{composition} of two such paths, and is defined naturally using the idea of juxtaposition. Let us give more details: let $\epsilon_i$ be a path of length 0 at the vertex $i$, and let $\gamma$ be any path over $Q$. Then $\epsilon_i \gamma$ is defined to be $\gamma$ if $s(\gamma) = i$, and defined to be zero otherwise. Analogously, $\gamma \epsilon_i$ is defined to be $\gamma$ if $e(\gamma) = i$, and defined to be zero otherwise. Also, let $\gamma = \gamma_1 \ldots \gamma_t$ and $\delta = \delta_1 \ldots \delta_s$ be two paths over $Q$. Then $\gamma \delta$ is defined to be the path $\gamma_1 \ldots \gamma_t \delta_1 \ldots \delta_s$ if $e(\gamma) = s(\delta)$, and defined to be zero otherwise.

With that multiplication, $kQ$ is a $k$-algebra, called the \textbf{path algebra} over the quiver $Q$. 
 
When necessary to distinguish the present concept of path algebra from the generalized one to be discussed below, we shall say that $kQ$ is the \textbf{ordinary path algebra} \index{path algebra!ordinary} over $Q$.

Since composition of paths is clearly associative, $kQ$ will be an associative algebra. Also, since we are assuming $Q$ to be finite, $kQ$ will have an identity element, given by

$$1_{kQ} = \sum_{i \in Q_0} \epsilon_i$$

Moreover, $kQ$ has finite dimension if and only if $Q$ is an acyclic quiver.

Since the arrows of $Q$ can be seen as elements of $kQ$, it makes sense to consider the ideal of $kQ$ generated by all arrows of $Q$; we denote this ideal by $J$. Note that, in the case where $Q$ is finite and acyclic, $J$ coincides with the Jacobson radical of $kQ$.

\begin{defi}
Let $I$ be an ideal of $kQ$. It is said to be \textbf{admissible} provided there is a natural number $n \geq 2$ such that $J^n \subseteq I \subseteq J^2$.
\end{defi}

Another important concept to recall here is that of relations on a quiver.

\begin{defi}
Given a quiver $Q$, an \textbf{(ordinary or usual) relation}\index{relation} over $Q$ is a $k$-linear combination of paths over $Q$, all of them having length greater than or equal to 2, and all of them starting and ending at the same vertex.
\end{defi}

\begin{obs}
It is a basic result that every admissible ideal of $kQ$ is generated by a finite set of relations. And reciprocally, if $Q$ is acyclic, every finite set of relations generates an admissible ideal of $kQ$.
\end{obs}

\begin{obs}

In practical situations, it is customary to define an algebra $A$ by giving a quiver $Q$ and a finite set of relations $R$ over $Q$. This means that $A$ is defined by $A = kQ/(R)$, where $(R)$ is the ideal generated by $R$. Sometimes it is also said that $A$ is the path algebra over $Q$ bound by $R$, or that $A$ is a \textbf{bound path algebra}.

\end{obs}

\subsection{\textit{Algebras as quotients of path algebras}}
\label{subsec:gabriel}

In this subsection we recall a well-known theorem due to P. Gabriel regarding path algebras, since the results that follow in this paper will use its ideas.

The theorem deals with the problem of assigning a quiver $Q_A$ to a given algebra $A$ in such a way that $A$ is a quotient of the path algebra $kQ_A$ by an admissible ideal. Formally, we have the following statement:

\begin{teo}
\label{th:gabriel_algebra}
Let $A$ be a finite-dimensional basic algebra over an algebraically closed field $k$. Then there is a quiver $Q_A$ and an admissible ideal $I$ of $kQ_A$ such that $A \cong kQ_A/I$. Moreover, $Q_A$ is uniquely determined by $A$.
\end{teo}

The quiver $Q_A$ is said to be the \textbf{Gabriel quiver}\index{quiver!Gabriel} or \textbf{ordinary quiver}\index{quiver!ordinary} of $A$.

\begin{obs} 
Let us just recall some ideas from the proof of Theorem~\ref{th:gabriel_algebra}, since they will be needed later. Let $E = \{e_1,\ldots,e_n\}$ be a complete set of pairwise orthogonal primitive idempotent elements of $A$. The set of vertices of $Q_A$ is taken to be $E$. Moreover, if $e_i,e_j \in E$, the number of arrows of $Q_A$ of the form $e_i \rightarrow e_j$ is equal to the natural number $\dimk \frac{e_i(\rad A)e_j}{e_i(\rad^2 A)e_j}$. Even though our definition of $Q_A$ apparently depends on the choice of the set $E$, it can be shown that another choice would produce a quiver which is just isomorphic to $Q_A$.
\end{obs}

\section{Generalized Path Algebras}
\label{sec:simplifying pa's}

\subsection{\textit{Definition of Generalized Path Algebra}}
\label{subsec:gpa}

The concept of generalized path algebra that we shall deal with here is the one introduced in a 2000 article by F. U. Coelho and S. X. Liu \cite{CLiu}. Besides giving the definition and basic properties, their interest there was on ring-theoretic properties, namely, analysing when a generalized path algebra is noetherian or prime. They also studied some uniqueness results (that is, what can we tell when two such algebras are isomorphic).

\begin{defi} 

Let $\Gamma= (\Gamma_0,\Gamma_1,s,e)$ be a quiver. Also, let $\mathcal{A}=(A_i)_{i \in \Gamma_0}$ be a family of $k$-algebras, one for each vertex of $\Gamma$.

\begin{itemize}

\index{$\mathcal{A}$-path}

\item An \textbf{$\mathcal{A}$-path of length 0} over $\Gamma$ is an element of the set 
$\bigcup_{i \in \Gamma_0} A_i$.

\item For $n > 0$, an \textbf{$\mathcal{A}$-path of length n} over $\Gamma$ is a sequence of the form
$$a_1 \beta_1 a_2 \ldots a_n \beta_n a_{n+1}$$
where $\beta_1 \ldots \beta_n$ is an ordinary path over $\Gamma$, $a_i \in A_{s(\beta_i)}$ if $i \leq n$, and $a_{n+1} \in A_{e(\beta_n)}$

\item We denote by $k[\Gamma,\mathcal{A}]$ the $k$-vector space spanned by all $\mathcal{A}$-paths over $\Gamma$.

\item The \textbf{generalized path algebra} over $\Gamma$ and $\mathcal{A}$ is the quotient vector space $k(\Gamma,\mathcal{A})=k[\Gamma,\mathcal{A}]/M$, where $M$ is the subspace generated by all elements of the form

$$(a_1 \beta_1 \ldots \beta_{j-1}(a^1_j+ \ldots+a^m_j)\beta_j a_{j+1} \ldots \beta_n a_{n+1}) - \sum_{l=1}^m (a_1 \beta_1 \ldots \beta_{j-1} a_j^l \beta_j \ldots \beta_n a_{n+1})$$

or, for $\lambda \in k$,

$$(a_1 \beta_1 \ldots \beta_{j-1}( \lambda a_j) \beta_j a_{j+1} \ldots \beta_n a_{n+1})- \lambda.(a_1 \beta_1 \ldots \beta_{j-1} a_j \beta_j a_{j+1} \ldots \beta_n a_{n+1})$$

\item The multiplication in $k(\Gamma,\mathcal{A})$ is induced by the of multiplication of the $A_i$'s and by composition of paths. Namely, it is defined by linearity and the following rule:

$$(a_1 \beta_1 \ldots \beta_n a_{n+1})(b_1 \gamma_1 \ldots \gamma_m b_{m+1}) = a_1 \beta_1 \ldots \beta_n (a_{n+1} b_1) \gamma_1 \ldots \gamma_m b_{m+1}$$

if $e(\beta_n) = s(\gamma_1)$, and 

$$(a_1 \beta_1 \ldots \beta_n a_{n+1})(b_1 \gamma_1 \ldots \gamma_m b_{m+1}) = 0 $$

otherwise.

\end{itemize}

\end{defi}

\begin{obs}

It should be easy to see that the ordinary path algebras are a particular case of generalized path algebras, simply by taking $A_i = k$ for every $i \in \Gamma_0$.

\end{obs}

\begin{obs}

In \cite{Kuls17}, J. Külshammer introduced a generalization of the concept of species, the so-called pro-species. As it is mentioned in that article, the generalized path algebras are tensor algebras over particular cases of pro-species, namely those which have algebras on each vertex and free bimodules on each arrow. We refer to \cite{Kuls17} for further details on this construction.

\end{obs}

Note that the generalized path algebra $k(\Gamma,\Aa)$ is an associative algebra. And since we are assuming the quivers to be finite, it also has an identity element, which is equal to $\sum_{i \in \Gamma_0} 1_{A_i}$. Finally, it is easy to observe that $k(\Gamma,\Aa)$ is finite-dimensional over $k$ if and only if so are the algebras $A_i$ and if $\Gamma$ is acyclic. 

\subsection{\textit{Generalized bound path algebras (gbp-algebras)}}
\label{subsec:gbpa}

In order to obtain the results that will follow, we need to extend our concept of generalized path algebras to allow them to have relations. In doing so, these algebras will be called \textbf{generalized bound path algebras}, here abbreviated as \textbf{gbp-algebras}. 

The idea of taking the quotient of a generalized path algebra by an ideal of relations has already been studied by Li Fang (see \cite{FLi1} for example). However, the concept we deal with here is slightly different.

\begin{defi}
\label{def:relation}

Let $k(\Gamma,\mathcal{A})$ be a generalized path algebra, where $\Gamma$ is a quiver and $\mathcal{A} = \{k\Sigma_i/\Omega_i : i \in \Gamma_0\}$ is a family of bound path algebras (here $\Sigma_i$ is a quiver and $\Omega_i$ is an admissible ideal of $k\Sigma_i$). 

Let $I$ be a finite set of relations over $\Gamma$ which generates an admissible ideal. Then we consider the following subset of $k(\Gamma,\Aa)$:

\begin{align*}
\Aa(I)&= \left\{ \sum_{i = 1}^t \lambda_i  \beta_{i1} \overline{\gamma_{i1}} \beta_{i2} \ldots \overline{\gamma_{i(m_i-1)}} \beta_{im_i} : \right.\\
&\left. \sum_{i = 1}^t \lambda_i \beta_{i1} \ldots \beta_{im_i} \text{  is a relation in } I 
\text{ and }\gamma_{ij}\text{ is a path in }\Sigma_{e(\beta_{ij})} \right\}
\end{align*}

Then the quotient $\frac{k(\Gamma,\mathcal{A})}{(\Aa(I))}$ is said to be a \textbf{generalized bound path algebra} (or \textbf{gbp-algebra}). To simplify the notation, we may also write $\frac{k(\Gamma,\mathcal{A})}{(\Aa(I))}=k(\Gamma,\Aa,I)$. When the context is clear, we may denote the set $\Aa(I)$ simply by $I$.
\end{defi}

\subsection{\textit{Realizing a gbp-algebra as a bound path algebra}}
\label{subsec:gen_icnlp}

Since a given gbp-algebra is an algebra in particular, it makes sense to apply Theorem \ref{th:gabriel_algebra} to it, and obtain an ordinary quiver bound by a set of relations. This is exactly the content of Theorem~\ref{th:gen_icnlp} below, which is the main result of this subsection.

Theorem~\ref{th:gen_icnlp} will be a generalization of a 2008 result by Ibáñez Cobos et al. \cite{ICNLP}. We recall the latter below. Roughly speaking, it describes the ordinary quiver with relations of a given generalized path algebra.

Let $\Lambda = k(\Gamma,\mathcal{A})$ be a generalized path algebra, with $\mathcal{A}=\{A_i : i \in \Gamma_0\}$. We will also suppose without loss of generality that $\Gamma_0 = \{1,2,\ldots,n\}$, in order to make notation simpler. 

Then, by Theorem~\ref{th:gabriel_algebra}, there is, for each $i$, a quiver $\Sigma_i$ and an admissible ideal $\Omega_i$ of $k\Sigma_i$ such that $A_i \cong k\Sigma_i/\Omega_i$.

In this context we define a quiver denoted by $\Gamma[\Sigma_1,\ldots,\Sigma_n]$ in the following way:

\begin{itemize}
\item The set of vertices of $\Gamma[\Sigma_1,\ldots,\Sigma_n]$ is $\bigcup_{i \in \Gamma_0} (\Sigma_i)_0$. 

\item If $a$ is a vertex of $\Sigma_i$ and $b$ is a vertex of $\Sigma_j$, then the number of arrows of the form $a \rightarrow b$ in $\Gamma[\Sigma_1,\ldots,\Sigma_n]$ is equal to the number of arrows of the form $a \rightarrow b$ in $\Sigma_i$ if $i = j$, and is equal to the number of arrows of the form $i \rightarrow j$ in $\Gamma$ if $i \neq j$.
\end{itemize}

As we are just about to see, the quiver $\Gamma[\Sigma_1,\ldots,\Sigma_n]$ coincides with the Gabriel quiver of the generalized path algebra $\Lambda$:

\begin{teo}(\cite{ICNLP},3.3) \label{th:icnlp}
With the hypothesis and notations from above, if $\Gamma$ is acyclic then
$$\Lambda \doteq k(\Gamma,\mathcal{A}) \cong \frac{\Gamma[\Sigma_1,\ldots,\Sigma_n]}{(\Omega_1,\ldots,\Omega_n)}$$
\end{teo}

\begin{obs}

In order to prove Theorem~\ref{th:gen_icnlp} below, we need to recall a few ideas from the proof of Theorem~\ref{th:icnlp}.

For every $i \in \Gamma_0$, let $E_i = \{e_i^1,\ldots,e_i^{c_i}\}$ be a complete set of pairwise orthogonal primitive idempotent elements of $A_i$. We remember from the proof of Theorem~\ref{th:gabriel_algebra} that we can assume that $(\Sigma_i)_0 = E_i$.

For every $i \in \Gamma_0$, $k\Sigma_i/\Omega_i \cong A_i$. So there is, for every $i$, a surjective algebra morphism $f_i: k\Sigma_i \rightarrow A_i$ such that $\Ker f_i = \Omega_i$.

With this notation, we have an algebra morphism $g: k\Gamma[\Sigma_1,\ldots,\Sigma_n] \rightarrow k(\Gamma,\Aa)$ that is uniquely determined by the following data:

\begin{itemize}
\item $g(\gamma) = f_i(\gamma)$ for every path $\gamma$ over $\Sigma_i$. (Observe that $\Sigma_i$ is a full subquiver of $\Gamma[\Sigma_1,\ldots,\Sigma_n]$).

\item $g(\alpha) = e_i^l \alpha e_j^m$ for every arrow $\alpha: e_i^l \rightarrow e_j^m$ such that $i \neq j$.
\end{itemize}

It is proved in \cite{ICNLP} that $g$ is a surjective algebra morphism and that $\Ker g = (\Omega_1,\ldots,\Omega_n)$. Clearly the statement of Theorem~\ref{th:icnlp} follows from this fact.

\end{obs}

Next, we are interested in proving Theorem~\ref{th:gen_icnlp} below, extending Theorem~\ref{th:icnlp} above to the context of the gbp-algebras defined in Subsection~\ref{subsec:gbpa}.

First we need to introduce some notations. Let $\Gamma$ be an acyclic quiver. In order to make the notation clearer, we may assume (without loss of generality) that $\Gamma_0=\{1,\ldots,n\}$ and that $\Gamma_1 = \{\alpha_1,\ldots,\alpha_m\}$. Let $\mathcal{A} = \{A_1,\ldots,A_n\}$, where each $A_i = k\Sigma_i/\Omega_i$ is a bound path algebra, where $\Sigma_i$ is a quiver and $\Omega_i$ is an admissible ideal of $k\Sigma_i$. Then $f_i: k\Sigma_i \rightarrow \frac{k\Sigma_i}{\Omega_i}$ will be the canonical projection.

Due to Theorem~\ref{th:icnlp}, there is a surjective algebra morphism $$g:kQ \rightarrow k(\Gamma,\mathcal{A})$$ where $Q = \Gamma[\Sigma_1,\ldots,\Sigma_n]$ is the quiver obtained from $\Gamma, \Sigma_1, \ldots, \Sigma_n$ as explained above. Also, $\Ker g = \Omega \doteq (\Omega_1,\ldots,\Omega_n)$. Now we denote by $c_i$ the number of vertices of $\Sigma_i$, and $c_{ij} = c_i.c_j$.

Next label the set of vertices of the $\Sigma_i$: $(\Sigma_i)_0 = \{e^1_i,\ldots,e^{c_i}_i\}$.
Now, by Theorem~\ref{th:icnlp}, if $\alpha_l:i \rightarrow j$ is an arrow of $\Gamma$, then there are $c_{ij}$ corresponding arrows in $Q$, which we shall denote by $\alpha_{l,e^p_i,e^q_j}:e^p_i \rightarrow e^q_j$, with $1 \leq p \leq c_i, 1 \leq q \leq c_j$. With this notation,

$$Q_1 = (\Sigma_1)_1 \cup \ldots \cup(\Sigma_n)_1 \cup \{\alpha_{l,x,y}:1 \leq l \leq m, x \in \Sigma_{s(\alpha_l)}, y \in \Sigma_{e(\alpha_l)}\}$$

Now we turn our attention to the relations. Suppose that the quiver $\Gamma$ has a given set $I$ of relations generating an admissible ideal in $k\Gamma$. Then, following Definition~\ref{def:relation}, we are going to consider the quotient of $k(\Gamma,\mathcal{A})$ by the ideal generated by $\Aa(I)$ below:

  \begin{align*}
  \Aa(I) \doteq &\left\{\sum_{1 \leq i \leq s} \lambda_i \beta_{i1} \overline{\gamma_{i1}} \beta_{i2} \ldots \overline{\gamma_{i(r-1)}} \beta_{ir}: \right.\\
  &\left. \gamma_{ij} \text{ is a path in } \Sigma_{e(\beta_{ij})}, \text{ and } \displaystyle\sum_{1 \leq i \leq s} \lambda_i \beta_{i1} \ldots \beta_{ir} \text{ is a relation in }I \right\}
  \end{align*}
  
  Define, in $kQ$,
  
  \begin{align}
  R(I) \doteq &\left\{\sum_{1 \leq i \leq s} \lambda_i (\beta_{i1,e_l^p,s(\gamma_{i1})}) \gamma_{i1} (\beta_{i2,e(\gamma_{i1}),s(\gamma_{i2})}) \gamma_{i2} \ldots \gamma_{i(r-1)} (\beta_{ir,e(\gamma_{i(r-1)}),e_{l'}^q} ): \right. \nonumber\\
  &\forall i, \lambda_i \in k, \gamma_{ij} \text{ is a path in } \Sigma_{e(\beta_{ij})} \text{ for } j\geq 1, \text{and } \displaystyle\sum_{1 \leq i \leq s} \lambda_i \beta_{i1} \ldots \beta_{ir} \nonumber\\
 &\left. \text{ is a relation in }I \text{ between vertices } l\text{ and } l',1\leq p \leq c_l, 1 \leq q \leq c_{l'}\right\} \label{eq:R_gen_icnlp}
  \end{align}
  
  And let $L(I)$ be the ideal generated by $R(I)$. Note that $L(I)$ is an ideal of  $kQ = k\Gamma[\Sigma_1,\ldots,\Sigma_n]$.
  
  \begin{defi}
  \label{def:induced_relations}
  
  The set $R(I)$ is the \textbf{set of relations in $k\Gamma[\Sigma_1,\ldots,\Sigma_n]$ induced} by the set of relations $I$ of $\Gamma$, and $L(I) \doteq (R(I))$ is the \textbf{ideal in $k\Gamma[\Sigma_1,\ldots,\Sigma_n]$ induced} by $I$.
  \end{defi}
  
\begin{lem}\label{lem:gen_icnlp}
Keep the notation from above. We have that $g(L(I)) = (\Aa(I))$.
\end{lem}
\begin{proof}

Remember how the surjection $g$ was defined in the proof of Theorem~\ref{th:icnlp}. We have that

\begin{align*}
&g\left(\sum_{1 \leq i \leq s} \lambda_i (\beta_{i1,e_l^p,s(\gamma_{i1})}) \gamma_{i1} (\beta_{i2,e(\gamma_{i1}),s(\gamma_{i2})}) \gamma_{i2} \ldots \gamma_{i(r-1)} (\beta_{ir,e(\gamma_{i(r-1)}),e_{l'}^q} )\right) \\
&=\sum_{1 \leq i \leq s} \lambda_i g(\beta_{i1,e_l^p,s(\gamma_{i1})}) f_{e(\beta_{i1})}(\gamma_{i1}) g(\beta_{i2,e(\gamma_{i1}),s(\gamma_{i2})}) f_{e(\beta_{i2})}(\gamma_{i2}) \ldots \\
& \hspace{4cm} f_{e(\beta_{i(r-1)})}(\gamma_{i(r-1)}) g(\beta_{ir,e(\gamma_{i(r-1)}),e_{l'}^q} ) \\ 
&=\sum_{1 \leq i \leq s} \lambda_i e_l^p \beta_{i1} e_{e(\beta_{i1})}^{s(\gamma_{i1})}\overline{\gamma_{i1}} e_{s(\beta_{i2})}^{e(\gamma_{i1})} \beta_{i2} e_{e(\beta_{i2})}^{s(\gamma_{i2})} \overline{\gamma_{i2}} \ldots \overline{\gamma_{i(r-1)}} e_{s(\beta_{ir})}^{e(\gamma_{i(r-1)})} \beta_{ir} e_{l'}^q\\ 
&=\sum_{1 \leq i \leq s} \lambda_i e_l^p\beta_{i1} \overline{\gamma_{i1}} \beta_{i2} \overline{\gamma_{i2}} \ldots \overline{\gamma_{i(r-1)}} \beta_{ir} e_{l'}^q \\
&=e_l^p \left( \sum_{1 \leq i \leq s} \lambda_i \beta_{i1} \overline{\gamma_{i1}} \beta_{i2} \overline{\gamma_{i2}} \ldots \overline{\gamma_{i(r-1)}} \beta_{ir} \right) e_{l'}^q
\end{align*}

The first conclusion is that $g(R(I)) \subseteq (\Aa(I))$. Since $g$ is an algebra morphism and $L(I)$ is the ideal generated by $R(I)$, this implies that $g(L(I)) \subseteq (\Aa(I))$. 

For the converse, remember that $\sum_p e_l^p = 1_{k\Sigma_l}$ and that $\sum_q e_{l'}^q = 1_{k\Sigma_{l'}}$. With this, the same calculations above show that $\Aa(I)$ is contained in the ideal generated by $g(R(I))$, and thus also contained in the ideal generated by $g(L(I))$, since $R(I) \subseteq L(I)$. But $g$ is surjective, so by the Correspondence Theorem $g(L(I))$ is already an ideal, which implies that $(\Aa(I)) \subseteq g(L(I))$.
\end{proof}

\begin{teo}\label{th:gen_icnlp}
Let $k(\Gamma,\Aa,I)$ be a gbp-algebra, with $\Gamma_0 =\{1,\ldots,n\}$ and with $\Aa = \{k\Sigma_1/\Omega_1,\ldots,k\Sigma_n/\Omega_n\}$ being a collection of bound path algebras. Then we have that: 

\begin{enumerate}

\item [(1)] $(\Omega_1,\ldots,\Omega_n)+L(I)$ is an admissible ideal of $k\Gamma[\Sigma_1,\ldots,\Sigma_n]$, and

\item [(2)] The following isomorphism holds:

$$\frac{k\Gamma[\Sigma_1,\ldots,\Sigma_n]}{(\Omega_1,\ldots,\Omega_n)+L(I)} \cong \frac{k(\Gamma,\mathcal{A})}{(\Aa(I))} \doteq k(\Gamma,\Aa,I)$$

\end{enumerate} 

\end{teo}

\begin{proof}

We already know that $\Omega \doteq (\Omega_1,\ldots,\Omega_n)$ is admissible from Theorem~\ref{th:icnlp}. Also $L(I)$ is admissible because it is generated by the set $R(I)$ in Equation~\ref{eq:R_gen_icnlp} above, and $R(I)$ is a set of relations which are sums of paths having length at least two, $I$ is admissible and $\Gamma$ is assumed acyclic and finite. It follows that $\Omega + L(I)$ is admissible.

Denote $J = (\Aa(I))$. Let $$\pi : k(\Gamma,\mathcal{A}) \rightarrow \frac{k(\Gamma,\mathcal{A})}{J}
$$ be the canonical projection. Define $$\tilde{\phi} \doteq \pi \circ g: k\Gamma[\Sigma_1,\ldots,\Sigma_n] \xrightarrow{g} k(\Gamma,\mathcal{A}) \xrightarrow{\pi} \frac{k(\Gamma,\mathcal{A})}{J}$$

Since $\tilde{\phi}$ is surjective, by the First Isomorphism Theorem it induces an isomorphism

$$\phi: \frac{k\Gamma[\Sigma_1,\ldots,\Sigma_n]}{\Ker \tilde{\phi}} \rightarrow \frac{k(\Gamma,\mathcal{A})}{J}$$

We claim that $\Ker \tilde{\phi} = g^{-1}(J)$. Indeed,

$$x \in \Ker \tilde{\phi} \Leftrightarrow \tilde{\phi}(x) = 0 \Leftrightarrow \pi \circ g(x) = 0 \Leftrightarrow g(x) \in J \Leftrightarrow x \in g^{-1}(J)$$

So it remains to prove that $g^{-1}(J) = \Omega + L(I)$.

$(\supseteq)$ Since $\Omega = \Ker g$, $g(\Omega) = 0 \subseteq J \Rightarrow \Omega \subseteq g^{-1}(J)$. From Lemma~\ref{lem:gen_icnlp}, $g(L(I)) = (\Aa(I)) \doteq J$, thus $L(I) \subseteq g^{-1}(J)$. Therefore $\Omega + L(I) \subseteq g^{-1}(J)$, because $g^{-1}(J) = \Ker \tilde{\phi}$ is an ideal and thus closed under sums.

$(\subseteq)$ Let $x \in g^{-1}(J)$ . Then $g(x) \in J$ and, by Lemma~\ref{lem:gen_icnlp}, there is an $l \in L(I)$ such that $g(x) = g(l)$. Then $x-l \in \Ker g = \Omega$, so there is $\omega \in \Omega$ such that $x-l = \omega$. Therefore $x = \omega + l$ with $\omega \in \Omega$ and $l \in L(I)$. Thus $x \in \Omega+L(I)$.

\end{proof}
  
\begin{ex}
\label{ex:gen_icnlp}
Let $\Lambda$ be the gbp-algebra given by the quiver

\begin{displaymath}
\xymatrix{
{\begin{matrix}
 k \\ \bullet \\ 1
 \end{matrix}}
\ar[rr]^{\alpha}="a" && 
{\begin{matrix}
 \frac{k\Sigma_2}{\Omega_2} \\ \bullet \\ 2
 \end{matrix}}
\ar@<0.5ex>[rr]^{\beta}="b" \ar@<-0.5ex>[rr]_{\gamma} &&
{\begin{matrix}
 k \\ \bullet \\ 3
 \end{matrix}}
\ar @{.} @/^2pc/ "a";"b"}
\end{displaymath}

with a relation $\alpha \beta = 0$, where $\Sigma_2$ is the quiver

\begin{displaymath}
\xymatrix{
\bullet \ar[d]^{\delta}\\
\bullet \ar[d]^{\varepsilon}\\ 
\bullet}
\end{displaymath}

and $\Omega_2 = (\delta \varepsilon)$. Applying Theorem~\ref{th:gen_icnlp}, we conclude that the Gabriel quiver $Q$ of $\Lambda$ is given by

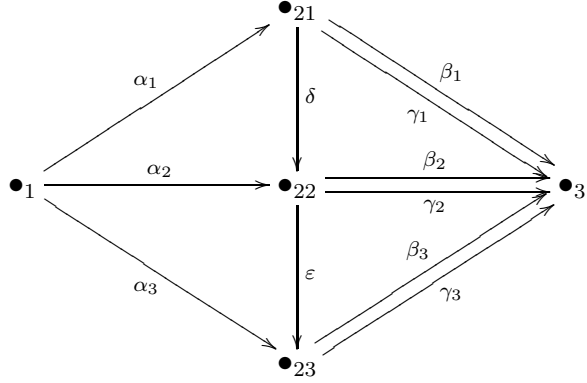
\begin{displaymath}
\xymatrix{
&&& \bullet_{21} \ar[dd]^{\delta} \ar@<0.5ex>[ddrrr]^{\beta_1} \ar@<-0.5ex>[ddrrr]_{\gamma_1}&&& \\
&&&&&&\\
\bullet_1 \ar[uurrr]^{\alpha_1} \ar[rrr]^{\alpha_2} \ar[ddrrr]_{\alpha_3} &&& \bullet_{22} \ar[dd]^{\varepsilon} \ar@<0.5ex>[rrr]^{\beta_2} \ar@<-0.5ex>[rrr]_{\gamma_2} &&& \bullet_3 \\
&&&&&&\\
&&& \bullet_{23} \ar@<0.5ex>[uurrr]^{\beta_3} \ar@<-0.5ex>[uurrr]_{\gamma_3} &&& }
\end{displaymath}

and that $\Lambda \cong kQ /( \Omega + L)$, where $\Omega = (\delta \varepsilon)$ and $$L = (\alpha_1 \beta_1, \alpha_1 \delta \beta_2, \alpha_1 \delta \varepsilon \beta_3, \alpha_2 \beta_2, \alpha_2 \epsilon \beta_3, \alpha_3 \beta_3)$$
\end{ex}

\section{Consequences and examples}
\label{subsec:conseq_icnlp}

In this section we are interested in getting consequences of Theorem~\ref{th:gen_icnlp}. This theorem showed how to obtain a \textbf{bound path algebra} isomorphic to a given \textbf{gbp-algebra}. What we want to analyse now is the inverse process:  how to obtain a gbp-algebra isomorphic to a given bound path algebra.

\begin{obs}
\label{obs:trivial_simplifications}
As already mentioned, there are always two trivial ways of realizing a bound path algebra as a gbp-algebra. Let $A$ be an algebra. Then:

\begin{itemize}

\item If $\Gamma= \bullet_1$, i.e., $\Gamma$ is a quiver with only one vertex and no arrows, and if $\Aa= \{A\}$, then obviously $A \cong k(\Gamma,\Aa)$.

\item By Theorem~\ref{th:gabriel_algebra}, there is a quiver $Q$ and an admissible ideal $I$ on $Q$ such that $A = kQ/I$. Let $\Aa = \{k:i \in Q_0\}$. Then clearly $A \cong k(Q,\Aa,I)$.

\end{itemize}

\end{obs}

\begin{defi}
\label{def:simplification}
\index{simplification}
\index{simplification!without relations}
\index{simplification!without cycles}
\index{simplification!trivial}
Let $A$ be an algebra. We say that a gbp-algebra $k(\Gamma,\Aa,I)$ is a \textbf{simplification} of $A$ if $A \cong k(\Gamma,\Aa,I)$. The two simplifications given above are called \textbf{trivial simplifications}. If $I = 0$, we say that the simplification is \textbf{without relations}. If $\Gamma$ is acyclic, we say that the simplification is \textbf{without cycles}.
\end{defi}

\begin{defi}
Let $k(\Gamma,\Aa,I)$ and $k(\Delta,\mathcal{B},J)$ be two simplifications of $A$, with $\Aa = \{A_i : i \in \Gamma_0\}$ and $\mathcal{B} = \{B_i : i \in \Delta_0\}$. We say that they are \textbf{equivalent}\index{simplification!equivalent} if there is an isomorphism of quivers $\phi:\Gamma \rightarrow \Delta$ such that $A_i \cong B_{\phi(i)}$ for every $i \in \Gamma_0$ and also if there is an isomorphism of algebras $k\Gamma/I \cong k\Delta/J$.
\end{defi}

\begin{defi}
\label{def:simplifiable}
\index{simplifiable}
\index{simplifiable!without relations}
\index{simplifiable!without cycles}
We say that an algebra $A$ is \textbf{simplifiable} if it admits a simplification which is not equivalent to the trivial ones listed in Remark~\ref{obs:trivial_simplifications}. We also use the terms \textbf{simplifiable without relations} or \textbf{without cycles} in the case where the non-trivial simplification is respectively without relations or without cycles.
\end{defi}

An example of a simplifiable algebra was given in Example~\ref{ex:gen_icnlp}, with the gbp-algebra form of the algebra $\Lambda$ being the (non-trivial) simplification. Also, in this particular case, the simplification is without cycles.

For the rest of this section, we establish some criteria to decide on whether or not a given algebra is simplifiable. Since Theorems~\ref{th:icnlp} and~\ref{th:gen_icnlp} only deal with the case when $\Gamma$ is acyclic, we shall only analyse simplifications without cycles.

\subsection{\textit{Equivalence relations over vertices}}

First we need to establish some useful terminology.

\begin{defi}
\label{def:rel_equiv}
Let $Q$ be a quiver and let $\sim \in Q_0 \times Q_0$ be an equivalence relation over the vertices of $Q$. 

\begin{itemize}

\item The \textbf{reduced quiver}\index{quiver!reduced} of $Q$, denoted by $Q^{\sim}$, is defined as the quiver obtained from $Q$ by deleting all the arrows whose start and end vertices are identified by $\sim$. Namely, $Q^{\sim}_0 = Q_0$ and $Q^{\sim}_1 = \{\alpha \in Q_1: s(\alpha) \nsim e(\alpha)\}$.

\item We say that $\sim$ is \textbf{coherent with the arrows}\index{coherent with arrows} of $Q$ if:

\begin{enumerate}

\item [(1)]For every arrow $i \rightarrow j$ contained in an oriented cycle of $Q$, we have that $i \sim j$.

\item [(2)]If $i \sim j$ and $j \nsim k$ then $[i,k] = [j,k]$ and $[k,i] = [k,j]$.

\end{enumerate}

\item Suppose that $\sim$ is coherent with the arrows. Then a \textbf{labelling}\index{labelling} of $Q$ will be a quiver morphism $z: Q^{\sim} \rightarrow \frac{Q^{\sim}}{\sim}$ such that $z(x) = \overline{x}$ for every vertex $x \in Q_0$, and such that, for every pair of vertices $x,y \in Q_0$, the restriction $z|_{Q^{\sim}(x,y)}: Q^{\sim}(x,y) \rightarrow \frac{Q^{\sim}}{\sim} (x,y)$ is bijective. (Since $\sim$ is coherent, such a $z$ will always exist). For the rest of this definition, we assume that $Q$ has a labelling denoted by $z$.

\item Let $\gamma$ be an ordinary path in $Q$. Note that we can write $\gamma = \delta_0 \alpha_1 \delta_1 \ldots \alpha_m \delta_m$, where, for each $i$, $\delta_i$ is a path passing only through vertices in the same equivalence class, and $\alpha_i: j_1 \rightarrow j_2$ is an arrow with $j_1 \nsim j_2$. Then the \textbf{path induced by $\gamma$ over the quotient quiver}\index{path!induced in the quotient} is defined to be the path $z(\gamma) = z(\alpha_1) z(\alpha_2)\ldots z(\alpha_m)$. If $\delta_0$ and $\delta_m$ have both length zero, we also say that $\gamma$ is a \textbf{straightforward} path.

\item Let $\gamma = \sum_{t=1}^r \lambda_t \gamma_t$ be a relation in $Q$, with $\lambda_t \in k$ and $\gamma_t$ a path in $Q$ for all $t$, with the $\gamma_t$ pairwise distinct. 

\begin{enumerate}

\item [(1)]We say that $\gamma$ is an \textbf{internal relation} \index{relation!internal} if $l(z(\gamma_t)) = 0$ for all $t$.

\item [(2)]If $\gamma_t$ is straightforward for every $t$ and if $z(\gamma_t) \neq z(\gamma_s)$ for every $1 \leq t,s \leq r$ such that $t$ and $s$ are distinct, then we say that $\gamma$ is an \textbf{external relation}. (Note that an external relation cannot be internal).

\end{enumerate}

\item Let $R$ be a finite set of relations in $Q$. We say $R$ and $\sim$ are \textbf{compatible}\index{relation!compatible}\index{compatible} if:

\begin{enumerate}

\item [(1)]Every relation in $R$ is either internal or external relative to $\sim$.

\item [(2)]If $\overline{x} \subseteq Q_0$ is an equivalence class relatively to $\sim$ and $\Sigma_{\overline{x}}$ denotes the full subquiver of $Q_0$ determined by $\overline{x}$, then the relations in $R$ involving only vertices of $\overline{x}$ generate an admissible ideal in $k \Sigma_{\overline{x}}$.

\item [(3)]If $\gamma = \sum_{t=1}^r \lambda_t \gamma_t \in R$ is external, where $\lambda_t \in k$ and $\gamma_t$ is a path in $Q$ for all $t$, then for every family of straightforward paths $\{\eta_t:1 \leq t \leq r\}$ in $Q$ such that for each $t$, $z(\gamma_t)=z(\eta_t)$, we have that $\sum_{t=1}^r \lambda_t \eta_t \in R$.

\end{enumerate}

\end{itemize}

\end{defi}

\begin{ex}
\label{ex:rel_equiv}

Let $A$ be the the bound path algebra given by the quiver $Q$ below:

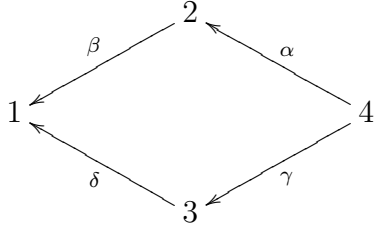
\begin{displaymath}
\xymatrix{
 & & 2 \ar[dll]_{\beta} & & \\
1 & & & & 4 \ar[ull]_{\alpha} \ar[dll]^{\gamma} \\
 & & 3 \ar[ull]^{\delta} & &
}
\end{displaymath}

bound by a set of relations $R$.
Let $\sim$ be the smallest equivalence relation over $Q_0$ such that $2 \sim 3$. Then $\sim$ is coherent with the arrows of $Q$.  The quotient quiver is given by 

\begin{displaymath}
\xymatrix{\overline{1} && \overline{2} \ar[ll]_{\eta} &&\overline{4} \ar[ll]_{\varepsilon}}
\end{displaymath}

In this case, there is only one possible labelling $z$ of $Q$. Let $R = \{\alpha \beta - \gamma \delta\}$. We have that $z(\alpha \beta) = z(\gamma \delta) = \varepsilon \eta$. This means that $\alpha \beta - \gamma \delta$ is neither internal nor external, and so $R$ is not compatible with $\sim$.

Nevertheless, the reader can verify that $R = \{\alpha \beta, \gamma \delta\}$ is compatible with $\sim$. This set contains two external relations.

\end{ex}

\begin{obs}
Observe that there is no ambiguity on the definition of compatibility since the given condition is defined on the set of generators and not on what they generate.

The necessity to label the arrows comes from the fact that a quiver might have multiple arrows between two vertices. If this does not happen, then there is only one labelling of the arrows and the concept of compatibility depends only on the set of relations.
\end{obs}

We shall see below (Theorems~\ref{th:simpl_from_equiv_rel} and~\ref{th:equiv_rel_from_simpl}) that there is some connection between simplifications and compatible equivalence relations. Roughly speaking, one can obtain one from the other. Analysing this connection will yield the promised criterion for simplifiability. We shall divide this study into the two parts below.

\subsection{\textit{Simplifications from equivalence relations}}
\label{subsubsec:simpl_from_rel}

We start with an algebra $A = kQ / (R)$, where $Q$ is a quiver, $R$ is a finite set of relations in $Q$ such that $(R)$ is admissible, and with an equivalence relation $\sim \subseteq Q_0 \times Q_0$ coherent with the arrows of $Q$ and compatible with $R$ relatively to a labelling $z$ of $Q$. Our aim is to produce a simplification without cycles $k(\Gamma,\Aa,I)$ of $A$.

The quiver $\Gamma$ will be the quotient quiver $\Gamma \doteq \frac{Q^{\sim}}{\sim}$, and thus the labelling $z$ is a quiver morphism $z: Q^{\sim} \rightarrow \Gamma$ which induces bijections $Q^{\sim}(x,y) \leftrightarrow \Gamma(x,y)$ for every pair of vertices $x,y \in Q_0$ such that $x \nsim y$.

We remark that $\Gamma$ defined this way is acyclic; indeed, if $\Gamma$ contains an oriented cycle, then it comes from an oriented cycle in $Q$, which is absurd, because the arrows of an oriented cycle in $Q$ must start and end at vertices in the same equivalence class, because of Definition~\ref{def:rel_equiv}.

Now write $\Gamma_0= \{\overline{x_1},\ldots,\overline{x_n}\}$, where the $x_i \in Q_0$ are all in distinct equivalence classes. We denote by $\Sigma_i$ the full subquiver of $Q$ whose vertices are those in $\overline{x_i}$.

By our hypothesis, we may write $R = R^{int} \sqcup R^{ext}$, where $R^{int}$ is composed only by internal relations, and $R^{ext}$, by external ones. Also, $R^{int} = \Omega_1 \sqcup \ldots \sqcup \Omega_n$, where $\Omega_i$ denotes the set of internal relations contained in $\Sigma_i$, for each $i$. Note that, due to Definition~\ref{def:rel_equiv}, $\Omega_i$ generates an admissible ideal in $\Sigma_i$. Then define $A_i = k\Sigma_i / (\Omega_i)$ and $\Aa = \{A_1, \ldots , A_n \}$.

Let $\gamma \in R^{ext}$. Then we may write 

$$\gamma = \displaystyle\sum_{t=1}^r \lambda_t \alpha_{t1} \delta_{t1} \alpha_{t2} \ldots \delta_{t(n_t-1)}\alpha_{tn_t}$$

where, for each $t$ and $i$,  $\lambda_t \in k$, $\delta_{ti}$ is a path whose vertices are all on the same equivalence class, and $\alpha_{ti}$ is an arrow between vertices on different equivalence classes. (Remember that, since $\gamma$ is external, all of its summands are straightforward paths). Then define a relation on $\Gamma$ in the following way:

$$z(\gamma) = \displaystyle\sum_{t=1}^r \lambda_t z(\alpha_{t1} \delta_{t1} \alpha_{t2} \ldots \delta_{t(n_t-1)}\alpha_{tn_t}) = \displaystyle\sum_{t=1}^r \lambda_t z(\alpha_{t1}) z(\alpha_{t2}) \ldots z(\alpha_{tn_t}) $$

Finally define $I = \{z(\gamma) : \gamma \in R^{ext}\}$. 

It remains to prove that this is in a fact a simplification of $A$, i.e., we have to prove that $A = kQ/(R) \cong k(\Gamma,\Aa,I)$.

If we apply Theorem~\ref{th:gen_icnlp} to the gbp-algebra $k(\Gamma,\Aa,I)$, one obtains that $k(\Gamma,\Aa,I) \cong \Gamma[\Sigma_1,\ldots,\Sigma_n]/(M)$ for a certain set of relations $M$. Also, it is easily observed that the quiver $\Gamma[\Sigma_1,\ldots,\Sigma_n]$ is isomorphic to $Q$, while $M$ can be taken as equal to $R^{int} \sqcup L(I)$. Summarizing, we have

$$k(\Gamma,\Aa,L) \cong \frac{kQ}{(M)} = \frac{kQ}{(R^{int} \sqcup L(I))} \text{, and } \frac{kQ}{(R)}=\frac{kQ}{(R^{int} \sqcup R^{ext})}$$

So it remains to prove that $R^{ext} = L(I)$.

$(\subseteq)$ Let $$\gamma = \displaystyle\sum_{t=1}^r \lambda_t \alpha_{t1} \delta_{t1} \alpha_{t2} \ldots \delta_{t(n_t-1)}\alpha_{tn_t} \in R^{ext}$$

where we use the same notation as above. Then 
$$z(\gamma) = \displaystyle\sum_{t=1}^r \lambda_t z(\alpha_{t1}) z(\alpha_{t2}) \ldots z(\alpha_{tn_t}) $$

And also $\delta_{ti} \in k\Sigma_{e(\alpha_{ti})}$ for all $i \geq 1$. By definition of $L(I)$ (Definition~\ref{def:induced_relations}), 

$$\gamma = \displaystyle\sum_{t=1}^r \lambda_t \alpha_{t1} \delta_{t1} \alpha_{t2} \ldots \delta_{t(n_t-1)}\alpha_{tn_t} \in L(I)$$

$(\supseteq)$ Again recall Definition~\ref{def:induced_relations}. Any element of $L(I)$ has the form 

$$\gamma = \displaystyle\sum_{t=1}^r \lambda_t \alpha_{t1} \eta_{t1} \alpha_{t2} \ldots \eta_{t(n_t-1)}\alpha_{tn_t}$$

where $\lambda_t \in k$, $\eta_{ti}$ is a path in $\Sigma_{e(\alpha_{ti})}$ for every $i \geq 1$, and

$$\displaystyle\sum_{t=1}^r \lambda_t z(\alpha_{t1}) \ldots z(\alpha_{tn_t})$$

is a relation in $I$. But, by the definition of $I$, this means that we have a relation

$$\gamma' = \displaystyle\sum_{t=1}^r \lambda_t \beta_{t1} \delta_{t1} \beta_{t2} \ldots \delta_{t(n_t-1)}\beta_{tn_t} \in R^{ext}$$

where $\delta_{ti}$ is a path in $\Sigma_{e(\beta_{ti})}$ for every $i \geq 1$ and, for every $t,i$, $z(\beta_{ti}) = z(\alpha_{ti})$.

Since, for every $t$, the $t$-th summands of $\gamma$ and $\gamma'$ have the same image through $z$, $\gamma' \in R^{ext}$ and $R$ is compatible with $\sim$, we have that $\gamma \in R^{ext}$. This finishes the proof.

To conclude this subsection, we summarize our findings:

\begin{defi}
With the notations introduced above, $A(\sim) \doteq k(\Gamma,\Aa,I)$ is called the \textbf{gbp-algebra associated} with the equivalence relation $\sim$ on the vertices of the Gabriel quiver of $A$.
\end{defi}

\begin{teo}
\label{th:simpl_from_equiv_rel}

Let $A = kQ/(R)$ be a bound path algebra with $R$ being a finite set of relations generating an admissible ideal and $\sim$ an equivalence relation coherent with the arrows and compatible
with $R$ relative to a fixed labelling of $Q$. Then $A(\sim)$ is a simplification without cycles of $A$.

\end{teo}

\subsection{\textit{Equivalence relations from simplifications}}

We start with an algebra $A = kQ / (R)$, where $Q$ is a quiver, $R$ is a finite set of relations in $Q$ such that $(R)$ is admissible, and a simplification without cycles $k(\Gamma,\Aa,I)$. Our aim is to obtain an equivalence relation $\sim \subseteq Q_0 \times Q_0$ coherent with the arrows of $Q$, a labelling of $Q$, and a finite set of relations $R'$ on $Q$ compatible with $\sim$ such that $(R')$ is admissible and $kQ / (R) \cong kQ / (R')$.

As a result of this, we will go through the reverse process of the previous subsection. Namely, we know that every simplification of $A$ induces an equivalence relation as above, a fact that helps to describe which simplifications the algebra $A$ may have.

Again we write $\Gamma_0 = \{1,\ldots,n\}$, $\Aa = \{A_1, \ldots, A_n\}$. Using Theorem~\ref{th:gabriel_algebra}, we may write $A_i = k\Sigma_i /(\Omega_i)$, where $\Sigma_i$ is a quiver and $\Omega_i$ is a finite set of relations on $\Sigma_i$ such that $(\Omega_i)$ is admissible. Applying Theorem~\ref{th:gen_icnlp} to the gbp-algebra $k(\Gamma,\Aa,I)$, we obtain that $Q = \Gamma[\Sigma_1,\ldots,\Sigma_n]$, because of the uniqueness in Theorem~\ref{th:gabriel_algebra}. In particular, $Q_0 = \bigsqcup_{i=1}^n (\Sigma_i)_0$. Since this union is disjoint, it is a partition of $Q_0$ and so defines an equivalence relation $\sim \subseteq Q_0 \times Q_0$. Recovering the notation from Subsection~\ref{subsec:gen_icnlp}, there is a labelling $z:Q^{\sim} \rightarrow \frac{Q^{\sim}}{\sim}$ such that $z(\alpha_{l,e_i^p,e_j^q}) = \alpha_l$ for every $\alpha_l: i \rightarrow j \in \Gamma_1$, $1 \leq p \leq c_i$, and $1 \leq q \leq c_j$. Also from Theorem~\ref{th:gen_icnlp}, we obtain a finite set $R' = \Omega_1 \sqcup\ldots \sqcup \Omega_n \sqcup R(I)$ of relations on $Q$ such that $(R')$ is admissible and $k(\Gamma,\Aa,I) \cong kQ/(R')$. Since by hypothesis $kQ / (R) \cong k(\Gamma,\Aa,I)$, we have that $kQ / (R) \cong kQ / (R')$. Next we verify:

\begin{itemize}

\item $\sim$ is coherent with the arrows of $Q$:

Since $\Gamma$ is acyclic, all the vertices in an occasional oriented cycle in $Q$ are identified by $\sim$. Let us check the second condition.

Take $x,y,z \in Q_0$ with $x \sim y$, $y \nsim z$. Then

$$[x,z]_Q = [\overline{x},\overline{z}]_{\Gamma} = [\overline{y},\overline{z}]_{\Gamma} = [y,z]_Q$$

The proof that $[z,x]_Q = [z,y]_Q$ is analogous.

\item $R'$ is compatible with $\sim$:

Let $$\gamma = \displaystyle\sum_{t=1}^r \lambda_t \delta_{t0} \alpha_{t1} \delta_{t1} \ldots \alpha_{tn_t} \delta_{tn_t} \in R'$$

where we maintain the notation used in Subsection~\ref{subsubsec:simpl_from_rel}. By the definition of $R'$, $\gamma$ is either a relation from $\Omega_1 \sqcup \ldots \sqcup \Omega_n$ or a relation from $R(I)$. 

In the former case, $n_t = 0$ for all $t$, which means that $\gamma$ is internal. In the latter case, note that $\delta_{t0}$ and $\delta_{tn_t}$ will both have length zero for every $t$. Also,
 
$$\displaystyle\sum_{t=1}^r \lambda_t z(\alpha_{t1}) \ldots z(\alpha_{tn_t})$$

will be a relation in $I$, where the paths $z(\alpha_{t1}) \ldots z(\alpha_{tn_t})$ can be assumed to be pairwise distinct (because we can always write the relations in $I$ in a way that this holds true). This proves that $\gamma$ is external in this case. Thus every relation in $R'$ is either internal or external and the first condition for compatibility in Definition~\ref{def:rel_equiv} is fulfilled.

The second condition is easily verified because, by construction, $(\Omega_i)$ is admissible in $k\Sigma_i$.

The last condition for compatibility is verified by observing that, because of the way $R'$ was defined, the set of external relations in $R'$ is equal to the set 

$$\{\displaystyle\sum_{t=1}^r \lambda_t \gamma_t: \gamma_t \text{ is straightforward for all }t \text{ and } \displaystyle\sum_{t=1}^r \lambda_t z(\gamma_t) \text{ is a relation in }I\}$$

\end{itemize}

As we did in Subsection~\ref{subsubsec:simpl_from_rel}, we shall summarize our discussion in the form of the following theorem.

\begin{teo}
\label{th:equiv_rel_from_simpl}
Let $A = kQ/(R)$ be a bound path algebra. Suppose that $A$ has a simplification without cycles $k(\Gamma,\Aa,I)$,  with $\Gamma_0 =\{1,\ldots,n\}$ and $\Aa= \{k\Sigma_1/(\Omega_1), \ldots, k\Sigma_n/(\Omega_n)\}$ being a collection of bound path algebras. Then:

\begin{enumerate}

\item[(1)] $Q_0 = \sqcup_{i = 1}^n (\Sigma_i)_0$, and this partition of $Q_0$ defines an equivalence relation $\sim$ on $Q_0$ which is coherent with the arrows of $Q$;

\item[(2)] $R' = \Omega_1 \sqcup \ldots \sqcup \Omega_n \sqcup R(I)$ is such that $A \cong k(\Gamma,\Aa,I)\cong kQ / (R')$, and 

\item[(3)] There is a labelling $z$ of $Q$ such that $R'$ is compatible (relatively to $z$) with the equivalence relation $\sim$.
\end{enumerate}
\end{teo}

\subsection{\textit{Some examples and applications}}

We shall devote this space to present some immediate applications of the criteria discussed in the previous subsections, and also show some practical examples.

\begin{obs}
Let $A$ be an algebra, and suppose $A = kQ / (R)$, where $Q$ is the Gabriel quiver of $A$ and $R$ is a finite set of relations over $Q$ generating an admissible ideal. By taking $\sim \subseteq Q_0 \times Q_0$ to be $Q_0 \times Q_0$ and applying Theorem~\ref{th:simpl_from_equiv_rel}, one obtains the first trivial simplification discussed in Remark~\ref{obs:trivial_simplifications}. In the case where $Q$ is acyclic, taking $\sim$ to be  $\{(x,x): x \in Q_0\}$ will yield the other trivial simplification.

\end{obs}

\begin{obs}
This remark is inspired by \cite{FLi2}, Section 6. Again, let $A = kQ/(R)$ be an algebra as in the previous remark, but assume that every arrow contained in a relation of $R$ is a loop and that the oriented cycles in $Q$ are all loops. We may thus write $R = \cup_{i \in \Gamma_0} \Omega_i$, where $\Omega_i$ is the set of relations of $R$ which involve only the vertex $i$. Let $\sim$ be the relation of equality in $Q_0$, that is, any vertex is identified with only itself through $\sim$. Again, this equivalence relation is coherent with the arrows and $R$ is compatible with it, every relation being internal. For every $i \in Q_0$, let $\Sigma_i$ be the set of loops in $Q$ based on $i$. Also let $\Gamma$ be the quiver obtained from $Q$ by deleting all loops, that is, $\Gamma_0 = Q_0$ and $\Gamma_1 = Q_1 \setminus \cup_{i \in Q_0} (\Sigma_i)_1$. Let $A_i = k\Sigma_i / (\Omega_i)$ and $\Aa = \{A_i : i \in \Gamma_0\}$. Then, again using Theorem~\ref{th:simpl_from_equiv_rel}, $A \cong k(\Gamma,\Aa)$. We have thus shown that $A$ has a simplification without any loop but with the same number of vertices of its Gabriel quiver. This simplification can only be trivial when $Q$ has only one vertex or when it does not have any loop.

Let us summarize this remark under the following statement:

\end{obs}

\begin{cor}
Let $A = kQ/(R)$, where $Q$ is a quiver whose only oriented cycles are loops, and $R$ is a finite set of relations, all of them contained in the loops of $Q$, such that $(R)$ is admissible. Then $A$ has a simplification without relations and without cycles $k(\Gamma,\Aa)$, where $\Gamma$ is the quiver obtained from $Q$ by deleting all loops.
\end{cor}

\begin{ex}
Remember the algebra of Example~\ref{ex:rel_equiv}. Using Theorem~\ref{th:simpl_from_equiv_rel}, if $R = \{\alpha \beta, \gamma \delta\}$, then the algebra has a simplification

\begin{displaymath}
\xymatrix{
k & & k^2 \ar[ll]_{\overline{\beta}} & & k \ar[ll]_{\overline{\alpha}}
}
\end{displaymath}

with the relation $\overline{\alpha} \overline{\beta} = 0$.

\end{ex}

\begin{ex}
Consider the algebra given by the quiver

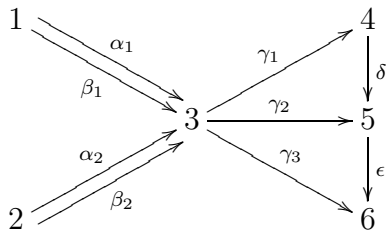
\begin{displaymath}
\xymatrix{
1 \ar@<1ex>[drr]^{\alpha_1} \ar[drr]_{\beta_1} & & & & 4 \ar[d]^{\delta} \\
 & & 3 \ar[rr]^{\gamma_2} \ar[urr]^{\gamma_1} \ar[drr]^{\gamma_3}& & 5 \ar[d]^{\epsilon} \\
2 \ar@<-1ex>[urr]_{\beta_2} \ar[urr]^{\alpha_2} & & & & 6
}
\end{displaymath}

bound by $\alpha_1 \gamma_1 = 0$, $\alpha_1 \gamma_2 = 0$, $\alpha_1 \gamma_3 = 0$, $\alpha_2 \gamma_1 = 0$, $\alpha_2 \gamma_2 = 0$, $\alpha_2 \gamma_3 = 0$, and $\delta \epsilon=0$.

Let $\sim$ be the smallest equivalence relation such that $1\sim2$ and $4 \sim 5 \sim 6$. Then $\sim$ is coherent with the arrows of the quiver. The quotient of the reduced quiver by $\sim$ is given below:

$$\xymatrix{
a \ar@<0.5ex>[rr]^{\alpha} \ar@<-0.5ex>[rr]_{\beta} & & b \ar[rr]^{\gamma} & & c
}$$

Then we can define a labelling $z$ by establishing $z(1) = z(2) = a$, $z(3) = b$, $z(4) = z(5) = z(6) = c$, $z(\alpha_1) = z(\alpha_2) = \alpha$, $z(\beta_1) = z(\beta_2) = \beta$ and $z(\gamma_1) = z(\gamma_2) = z(\gamma_3) = \gamma$.

With this labelling, the set of relations given above is compatible with $\sim$. Indeed, $\delta \epsilon$ is an internal relation and the relations of the form $\alpha_i \gamma_j$ are external relations, all of them having the induced path $\alpha \gamma$ on the quotient. Moreover, every straightforward path whose induced path is $\alpha \gamma$ is one of the $\alpha_i \gamma_j$. So this algebra is isomorphic to the gbp-algebra

$$\xymatrix{
k^2 \ar@<0.5ex>[rr]^{\alpha} \ar@<-0.5ex>[rr]_{\beta} & & k \ar[rr]^{\gamma} & & A
}$$

bound by $\alpha \gamma = 0$, where $A$ is the algebra given by the quiver 

$$\xymatrix{
4 \ar[d]^{\delta} \\
5 \ar[d]^{\epsilon} \\
6
}$$

bound by $\delta \epsilon =0$.

\end{ex}

\section{Acknowledgements}
This paper was written while the first named author was a graduate student under the supervision of the second. The authors gratefully acknowledge financial support by São Paulo Research Foundation (FAPESP), grant \#2018/18123-5. The second author has also a grant by CNPq (Pq 312590/2020-2).


\begin{thebibliography}{10}
\bibitem{AC} 
Assem, I., Coelho, F. U. (2020). {\it Basic Representation Theory of Algebras}. Graduate Texts in Mathematics, Vol. 283. Springer.

\bibitem{ASS} 
Assem, I., Simson, D., Skowro\'n{}ski, A. (2006). {\it Elements of the Representation Theory of Associative Algebras}, Vol. 1. London Mathematical Society Student Texts, Vol. 65. Cambridge University Press. 

\bibitem{ARS}
Auslander, M., Reiten, I., Smal\o{}, S. (1995). {\it Representation Theory of Artin algebras} Cambridge Studies in Advanced Mathematics, Vol. 36. Cambridge University Press.

\bibitem{Chust2}
Chust, V., Coelho, F. U. Representations and homological invariants of generalized bound path algebras. (In preparation).

\bibitem{ICNLP}   
Ibánez Cobos, R. M.,  Navarro, G., López Peña, J. (2008). A note on generalized path algebras. 
{\it Rev. Roumaine Math. Pures Appl.} 53(1): 25-36.

\bibitem{CLiu} 
Coelho, F. U., Liu, S-X. (2000). Generalized path algebras. In: {\it Interactions between ring theory and representations of algebras (Murcia)}, pp. 53-66. {\it Lecture Notes in Pure and Appl. Math.} 210. Dekker, New York.

\bibitem{Kuls17} Külshammer, J. (2017). Pro-species of algebras I: Basic Properties. {\it Algebr. Represent. Theor.} 20: 1215-1238.

\bibitem{FLi1} Li, F. (2007). Characterization of left artinian algebras through pseudo path algebras. {\it J. Aust. Math. Soc.} 83: 385-416.

\bibitem{FLi2} Li, F. (2012). Modulation and natural valued quiver of an algebra. \textit{Pacific J. Math.} 256(1): 105-128.

\end{thebibliography}
\end{document}